\documentclass[12pt, reqno]{amsart}
\usepackage{amsthm}
\usepackage{amsmath,amssymb,amsthm,mathrsfs,amscd}
\usepackage{MnSymbol}
\usepackage{color}
\usepackage[all]{xy}
\usepackage[english]{babel}
\usepackage{graphicx}
\usepackage[latin1]{inputenc}
\usepackage[T1]{fontenc}
\usepackage[final=true]{hyperref}
\usepackage{tikz}
\usetikzlibrary{patterns}
\usetikzlibrary{arrows}
\usepackage{rotating}
\usepackage{cases}
\usepackage{xkeyval}

\theoremstyle{plain}
\newtheorem{thm}{Theorem}[section]
\newtheorem*{thm*}{Theorem}

\newtheorem{theorema}{Theorem}

\newtheorem{lem}[thm]{Lemma}
\newtheorem{cor}[thm]{Corollary}
\newtheorem{prop}[thm]{Proposition}

\theoremstyle{definition}
\newtheorem{ex}[thm]{Example}
\newtheorem{rem}[thm]{Remark}
\theoremstyle{definition}
\newtheorem{defi}[thm]{Definition}

\newcommand{\sym}[2]{\{ #1 , #2\}}

\DeclareMathOperator{\N}{\mathbb{N}}

\DeclareMathOperator{\Z}{\mathbb{Z}}

\DeclareMathOperator{\C}{\mathbb{C}}

\DeclareMathOperator{\sgn}{sgn}

\DeclareMathOperator{\Hom}{Hom}

\DeclareMathOperator{\gr}{gr}

\DeclareMathOperator{\iot}{NA}
\DeclareMathOperator{\CA}{CA}

\newcommand{\dual}[1]{\widehat{#1}}

\DeclareMathOperator{\pp}{p}

\DeclareMathOperator{\CAi}{CA_{\ii}}

\newcommand{\CAid}{\dual{\CA}_{\ii}}

\DeclareMathOperator{\gCAi}{\gr CA_{\ii}}
\DeclareMathOperator{\gCAj}{\gr CA_{\jj}}
\newcommand{\gCAid}{\gr \dual{\CA}_{\ii}}

\newcommand{\gCAjd}{\gr\dual{\CA}_{\jj}}
\DeclareMathOperator{\gioti}{\gr\iot_{\ii}}

\newcommand{\giotid}{\gr\dual{\iot}_{\ii}}
\newcommand{\giotjd}{\gr\dual{\iot}_{\jj}}

\DeclareMathOperator{\NAi}{\iot_{\ii}}

\newcommand{\NAid}{\dual{\iot}_{\ii}}

\DeclareMathOperator{\w0}{w_0}

\DeclareMathOperator{\trl}{\Phi}
\DeclareMathOperator{\trs}{\Psi}

\DeclareMathOperator{\stp}{str}

\DeclareMathOperator{\ii}{\mathbf{i}}
\DeclareMathOperator{\jj}{\mathbf{j}}

\DeclareMathOperator{\trop}{trop}

\DeclareMathOperator{\W}{\mathcal{W}(w_0)}

\DeclareMathOperator{\etaa}{\mathfrak{l}}
\DeclareMathOperator{\luslpotia}{\etaa_{\ii, a}}
\DeclareMathOperator{\lusrpotia}{\etaa_{\ii, -a}}
\DeclareMathOperator{\luslpotja}{\etaa_{\jj, a}}
\DeclareMathOperator{\lusrpotja}{\etaa_{\jj, -a}}
\DeclareMathOperator{\luslpotiald}{\check{\etaa}_{\ii, a}}
\DeclareMathOperator{\lusrpotiald}{\check{\etaa}_{\ii, -a}}

\DeclareMathOperator{\lusrpotjald}{\check{\etaa}_{\jj, -a}}

\newcommand{\lusrpot}[2]{\etaa_{#1, -#2}}

\newcommand{\lusrpotld}[2]{\check{\etaa}_{{#1}, -{#2}}}

\DeclareMathOperator{\zetaa}{\mathfrak{s}}
\DeclareMathOperator{\stlpotia}{\zetaa_{\ii, a}}
\DeclareMathOperator{\strpotia}{\zetaa_{\ii, -a}}
\DeclareMathOperator{\stlpotja}{\zetaa_{\jj, a}}
\DeclareMathOperator{\strpotja}{\zetaa_{\jj, -a}}
\DeclareMathOperator{\stlpotiald}{\check{\zetaa}_{\ii, a}}
\DeclareMathOperator{\strpotiald}{\check{\zetaa}_{\ii, -a}}

\newcommand{\stlpot}[2]{\zetaa_{#1, #2}}

\DeclareMathOperator{\GN}{G/\mathcal{N}}

\DeclareMathOperator{\starop}{\star op}

\DeclareMathOperator{\Lg}{^{L}\mathfrak{g}}

\DeclareMathOperator{\TT}{\mathbb{T}}

\DeclareMathOperator{\TTil}{{\mathbb{L}}_{\ii}}
\DeclareMathOperator{\TTjl}{{\mathbb{L}}_{\jj}}
\newcommand{\TTis}{{\mathbb{S}}_{\ii}}

\DeclareMathOperator{\TTilp}{\mathbb{L}_{\jj}}
\DeclareMathOperator{\gTTil}{\gr\TTil}
\DeclareMathOperator{\gTTis}{\gr{{\mathbb{S}}_{\ii}}}
\DeclareMathOperator{\gTTjl}{\gr\TTjl}
\DeclareMathOperator{\gTTjs}{\gr{{\mathbb{S}}_{\jj}}}

\newcommand{\TTjs}{{\mathbb{S}}_{\jj}}

\newcommand{\xv}[1]{\text{x}_{#1}}
\newcommand{\xvc}[2]{\text{x}_{#1, #2}}

\newcommand{\TTix}{\dual{\mathbb{T}}_{\ii}}

\newcommand{\TTixp}{\dual{\mathbb{T}}_{\jj}}
\DeclareMathOperator{\TTia}{\mathbb{T}_{ \ii}}
\DeclareMathOperator{\TTiap}{\mathbb{T}_{\jj}}
\newcommand{\TTx}{\dual{\mathbb{T}}}
\DeclareMathOperator{\TTse}{\mathbb{T}_{\Sigma}}
\DeclareMathOperator{\TTsea}{\mathbb{T}_{ \Sigma}}
\newcommand{\TTsex}{\dual{\mathbb{T}}_{\Sigma}}
\newcommand{\TTsexp}{\dual{\mathbb{T}}_{\Sigma'}}
\DeclareMathOperator{\SE}{\Sigma}
\DeclareMathOperator{\SEp}{\Sigma'}
\DeclareMathOperator{\TTsep}{\mathbb{T}_{\Sigma'}}

\newcommand{\ttsx}{\dual{\mu}_{\SEp}^{\SE}}
\DeclareMathOperator{\ttsa}{{\mu}_{\SEp}^{\SE}}
\newcommand{\ttix}{\dual{\mu}_{\jj}^{\ii}}
\DeclareMathOperator{\ttia}{{\mu}_{\jj}^{\ii}}

\newcommand{\xse}[1]{x^{\SE}_{#1}}

\newcommand{\trod}[1]{\left[#1 \right]_{\trop}}
\newcommand{\tro}[1]{[#1]_{\trop}}

\DeclareMathOperator{\lci}{\mathcal{L}_{\ii}}

\DeclareMathOperator{\glci}{\gr\mathcal{L}_{\ii}}
\DeclareMathOperator{\glcj}{\gr\mathcal{L}_{\jj}}
\DeclareMathOperator{\glcild}{\gr\mathcal{L}^{\vee}_{\ii}}

\DeclareMathOperator{\sci}{\mathcal{S}_{\ii}}

\DeclareMathOperator{\gsci}{\gr\mathcal{S}_{\ii}}

\DeclareMathOperator{\gscild}{\gr\mathcal{S}^{\vee}_{\ii}}

\newcommand{\xcs}{\dual{\mathcal{C}}_{{\Sigma}}}
\newcommand{\xci}{\dual{\mathcal{C}}_{{\ii}}}

\newcommand{\acs}{{\mathcal{C}}_{{\Sigma}}}
\newcommand{\aci}{{\mathcal{C}}_{{\ii}}}

\newcommand\restr[2]{{
  \left.\kern-\nulldelimiterspace 
  #1 
  \vphantom{\big|}
  \right|_{#2}
  }}
  \makeatletter
\renewcommand*\env@cases[1][1.2]{
  \let\@ifnextchar\new@ifnextchar
  \left\lbrace
  \def\arraystretch{#1}
  \array{@{}l@{\quad}l@{}}
}
\makeatother

\begin{document}

\title[Polyhedral parametrizations \& cluster duality]{Polyhedral parametrizations of \\ canonical bases \& cluster duality}

\author{Volker Genz}
\address{Mathematical Institute, Ruhr-Universit{\"a}t Bochum}
\email{volker.genz@gmail.com}

\author{Gleb Koshevoy}
\address{CEMI Russian Academy of Sciences, MCCME and Poncelet laboratory (CNRS and IMU)}
\email{koshevoy@cemi.rssi.ru}

\author{Bea Schumann}
\address{Mathematical Institute, University of Cologne}
\email{bschuman@math.uni-koeln.de}

\begin{abstract} 
We establish the relation of the potential function constructed by Gross-Hacking-Keel-Kontsevich's and Berenstein-Kazhdan's decoration function on the open double Bruhat cell in the base affine space $\GN$ of a simple, simply connected, simply laced algebraic group $G$. As a byproduct we derive explicit identifications of polyhedral parametrization of canonical bases of the ring of regular functions on $\GN$ arising from the tropicalizations of the potential and decoration function with the classical string and Lusztig parametrizations. 
\end{abstract}

\maketitle

\section*{introduction}
Let $G$ be a simple, simply connected, simply laced algebraic group over $\mathbb{C}$, $B\subset G$ a Borel subgroup with unipotent radical $\mathcal{N}$.

There is a cluster space $\mathcal{A}$ (\cite{BFZ2}) and a dual cluster space $\mathcal{X}$ (\cite{FG}) associated to the open double Bruhat cell in the base affine space $\GN$. We are interested in the following functions both playing a crucial role in the study of canonical vector space bases of the ring of regular functions $H^0(\GN,\mathcal{O}_{\GN})$ on $\GN$.

On the one hand Berenstein-Kazhdan's decoration function $f^B$ defined in \cite{BK2}, a regular function on $\mathcal{A}$. The decoration function is a crucial part of the construction of a decorated geometric crystal and thus intimately connected to the canonical basis $\mathcal{B}_{\text{can}}$ of $H^0(\GN,\mathcal{O}_{\GN})$ independently constructed by Kashiwara and Lusztig's \cite{Lu, Ka}.

On the other hand, a remarkable vector space basis $\mathbb{B}$ was recently constructed (up to a natural conjecture, see Remark \ref{GHKKrem}) by Gross-Hacking-Keel-Kontsevich \cite{GHKK} using methods in mirror symmetry. An important ingredient in the construction of $\mathbb{B}$ is a regular function $W$ on $\mathcal{X}$ which we call the GHKK-potential.

We relate the GHKK-potential to the decoration function as follows.

\begin{theorema}\label{thma} There exists a regular map $\varphi: \mathcal{A} \rightarrow \mathcal{X}$ such that 
$${f^B} = {W}\circ \varphi.$$
\end{theorema}

The cluster spaces $\mathcal{A}$ and $\mathcal{X}$ are unions of open tori $\mathcal{A}=\bigcup_{\Sigma}\TTsea$, $\mathcal{X}=\bigcup_{\Sigma}\TTsex$, which are glued via certain birational transformations, called $\mathcal{A}$- and $\mathcal{X}$-cluster mutations, respectively. The elements $\SE$ in the common index set of the two dual toric systems are called seeds. The families of charts equip $\mathcal{A}$ and $\mathcal{X}$ with the structure of a positive variety admitting tropicalization.

The functions $f^B$ and $W$ lead to polyhedral parametrization of $\mathcal{B}_{\text{can}}$ and $\mathbb{B}$, respectively, one for each seed $\SE$. By \cite{GHKK} the integer polyhedral cone
$$ \xcs=\{x \in [\TTsex]_{trop} \mid [\restr{W}{\TTsex}]_{trop}(x)\ge 0\}$$
parametrizes $\mathbb{B}$. By \cite{BK2} the tropicalization of the decoration function $f^B$ cuts out an integer polyhedral cone 
$$\acs =\{x \in [\TTsea]_{trop} \mid [\restr{f^B}{\TTsea}]_{trop}(x)\ge 0\}$$
which parametrizes Lusztig's canonical basis of the base affine space of the Langlands dual group of $G$.

Theorem \ref{thma} is deduced by studying the interplay of Gross-Hacking-Keel-Kontsevich's polyhedral parametrization $\xcs$ and the parametrization arising from the tropicalization of the Berenstein-Kazhdan decoration function $\acs$ with classical parametrizations of Lusztig's canonical basis obtained by Lusztig and Kashiwara. 

Both Lusztig's and Kashiwara's construction yield a family of polyhedral parametrizations, one for each reduced word $\mathbf i$ of the longest element $\w0$ of the Weyl group of $G$, by the graded string cone $\gr\mathcal{S}_{\mathbf i}$ and the graded cone of Lusztig's parametrization $\gr\mathcal{L}_{\mathbf i}$, respectively. We related $\gsci$ and $\glci$ to the functions $f^B$ and $W$ by introducing regular maps $\zetaa_{\ii}$ and $\etaa_{\ii}$ satisfying
\begin{align*}
\gr\mathcal{S}_{\mathbf i}&=\{x\in [T]_{trop} \mid [\zetaa_{\ii}]_{trop}(x) \ge 0\}, \\
\gr\mathcal{L}_{\mathbf i}&=\{x\in [T]_{trop} \mid [\etaa_{\ii}]_{trop}(x) \ge 0\}.
\end{align*}
We denote the corresponding objects for the Langlands dual group of $G$ by $\zetaa^{\vee}_{\ii}$, $\etaa^{\vee}_{\ii}$, $\gr\mathcal{S}^{\vee}_{\mathbf i}$ and $\gr\mathcal{L}^{\vee}_{\mathbf i}$

The interplay between the various parametrizations can be summarized in the following theorem.

\begin{theorema}[Theorem \ref{lempot}, Theorem \ref{string2}, Lemma \ref{unicones}]\label{thmb} For every reduced word $\mathbf i$ there are toric charts $\TTia$ and $\TTix$ of $\mathcal{A}$ and $\mathcal{X}$ and explicit isomorphisms $\gCAi$, $\gioti$, $\gCAid$ and $\giotid$ with
	\begin{align*}
	\zetaa_{\ii}&=\restr{W}{\TTix} \circ \gCAid,&\etaa_{\ii}&=\restr{W}{\TTix} \circ \giotid, \\
	\restr{f^B}{T_{\ii}} &=\zetaa^{\vee}_{\ii}\circ \gr\NAi,&\restr{f^B}{T_{\ii}}&=\etaa^{\vee}_{\ii}\circ \gr\CAi.
	\end{align*}
Furthermore, we obtain the following family of commuative diagrams of linear maps indexed by reduced words
$$ \xymatrix{ \gscild \ar@{^{(}->}[d] & & \aci\ar[ll]_{\tro{\gioti}}^{\widetilde{}} \ar[rr]^{\tro{\gCAi}}_{\widetilde{}} \ar@{^{(}->}[d] & & \glcild \ar@{^{(}->}[d]\\
	\gsci\ar[rr]^{\tro{\gCAid}}_{\widetilde{}}  & & \xci   & & \glci\ar[ll]_{\tro{\giotid}}^{\widetilde{}}. 
}$$
\end{theorema}

We obtain Theorem \ref{thma} using Theorem \ref{thmb} and a result  of Zelevinsky stating that $\mathcal{A}$ is covered by the tori associated to reduced words up to codimension $2$. Note that there are two candidates for the map $\varphi$ arising from Theorems \ref{thmb}. We show that both candidates coincide.

Another consequence of Theorem \ref{thmb} is a lattice isomorphism from the graded string cone to the graded cone of Lusztig's parametrization, recovering a result of Caldero-Marsh-Morier-Genoud. 

There are two natural types of inequalities for both $\gsci$ and $\glci$: One type yields the inequalities for a polyhedral parametrization of a canonical basis of the ring $H^0(\mathcal{N},\mathcal{O}_{\mathcal{N}})$ of regular functions on the unipotent radical.
We call these inequalities the cone inequalities for the sake of this introduction. The other type of inequalities describe the graded lift of a polyhedral parametrization of $H^0(\mathcal{N},\mathcal{O}_{\mathcal{N}})$ 
to a polyhedral parametrization of a canonical basis of
 $H^0(\GN,\mathcal{O}_{\GN})$, called here grading inequalities. 

We show that under Caldero-Marsh-Morier-Genoud's map the cone inequalities of the graded cone of Lusztig's parametrization is mapped to the grading inequalities of the graded string cone and vise versa. We further give an affine unimodular map between the corresponding weight polytopes.

In certain cases, polyhedral parametrizations of canonical bases of rings of regular functions on a variety $X$ by a cone lead to flat degenerations of $X$ to the toric variety defined by the cone. In the case of flag varieties, an overview of many such cases is given in \cite{FFL2}. 

Theorem \ref{thmb} implies that, in the case of the base affine space, the toric fibers appearing in the degeneration construction by Caldero (\cite{C}) and Alexeev-Brion (\cite{AB}) also appear in the degenerations constructed by Gross-Hacking-Keel-Kontsevich (\cite{GHKK}). In the special case of $G=\text{SL}_n(\mathbb{C})$ this was proven previously in \cite{BF}.

Moreover, toric degenerations associated to the graded cone of Lusztig's parametrization and the graded string cone where constructed in \cite{FFL1}. Hence Theorem \ref{thmb} provides further evidence that there should be a natural connection between a subclass of Fang-Fourier-Littelmann's toric degenerations constructed in op. cit. and Gross-Hacking-Keel-Kontsevich's toric degenerations constructed by cluster duality (see \cite[10.1]{FFL2}).

\section*{Acknowledgment}
We would like to thank Peter Littelmann for several valuable explanations and discussions. Further we would like to thank Xin Fang, Valentin Rappel and Christian Steinert for helpful comments. This work is part of a project in the SFB/TRR 191 `Symplectic Structures in Geometry,
Algebra and Dynamics' from which the authors gratefully acknowledge (partial) financial support. G.K. was further supported by
the grant RSF 16-11- 10075.

\section{Background and Notations}

\subsection{Simply-laced Lie algebras}
Let $\mathfrak{g}$ be simple, simply laced complex Lie algebra of rank $n$, $I:=[n]:=\{1,\ldots,n\}$, $C=(c_{i,j})_{i,j\in I}$ its Cartan matrix and $\mathfrak{h}\subset \mathfrak{g}$ a Cartan subalgebra.
We chose simple coroots $\{h_a\}_{a\in I}\in\mathfrak{h}$ and simple roots $\Delta^+=\{ \alpha_a \}_{a\in I}\subset \mathfrak{h}^*$ with $\alpha_a (h_b)= c_{a,b}$
and denote by $\Delta^+\subset\mathfrak{h}^*$ the positive roots associated to $\{\alpha_a\}$.

The fundamental weights $\{\omega_a\}_{a\in I}\subset\mathfrak{h}^*$ of $\mathfrak{g}$ are given by $\omega_a(h_j)=\delta_{a,j}$. We denote by $P=\langle \omega_a \mid a\in[n] \rangle_{\mathbb{Z}}$ the weight lattice of $\mathfrak{g}$ and by $P^+=\langle \omega_a \mid a\in[n] \rangle_{\mathbb{N}}\subset P$ the set of dominant weights.

The Langlands dual Lie algebra $^L\mathfrak{g}$ of $\mathfrak{g}$ is the simple, simply laced complex Lie algebra with Cartan matrix $C$, Cartan subalgebra $\mathfrak{h}^*$, simple roots $\{h_a\}_{a\in I}$, simple coroots $\{\alpha_a\}_{a\in I}$ and $h_a (\alpha_b)=c_{a,b}$.
The fundamental weights of $^L\mathfrak{g}$ are $\{\omega^{\vee}_a\}_{a\in I}\subset \mathfrak{h}$ where $\alpha_a (\omega^{\vee}_b) = \delta_{a,b}$.

\subsection{Weyl groups and reduced words}
The Weyl group $W$ of $\mathfrak{g}$ is a Coxeter group generated by the simple reflections $s_a$ ($a \in I$) with relations 
\begin{align*} s_i^2 &=id, \\
 s_{i_1}s_{i_2} & =s_{i_2}s_{i_1} \qquad \, \, \, \text{if }c_{i_1,i_2}=0 \quad \text{ (commutation relation)},  \\
s_{i_1}s_{i_2}s_{i_1}&=s_{i_2}s_{i_1}s_{i_2} \quad \text{ if }c_{i_1,i_2}=-1 \quad \text{ (braid relation)}.
\end{align*}

The group $W$ has a unique longest element $\w0$ of length $N=\#\Delta^+$. For a reduced expression $s_{i_1} \cdots s_{i_N}$ of $\w0$ we write $\ii:=(i_1,\ldots, i_N)$ and call $\ii$ a \emph{reduced word} (for $\w0$). The set of reduced words for $\w0$ is denoted by $\mathcal{W}(\w0)$.

We have two operations on the set of reduced words $\W$.
\begin{defi}
A reduced word $\jj=(j_1,\ldots,j_N)$ is defined to be obtained from $\ii=(i_1,\ldots,i_N)\in \W$ by a \emph{$2$-move at position $k\in [N-1]$} if $i_{\ell}=j_{\ell}$ for all $\ell\notin \{k,k+1\}$, $(i_{k+1}, i_{k})=(j_{k},j_{k+1})$ and $c_{i_{k},i_{k+1}}=0$.

A reduced word $\jj=(j_1,\ldots,j_N)$ is defined to be obtained from $\ii=(i_1,\ldots,i_N)\in \W$ by a \emph{$3$-move at position $k\in [N-1]$} if $i_{\ell}=j_{\ell}$ for all $\ell\notin \{k-1,k,k+1\}$, $j_{k-1}=j_{k+1}=i_k$, $j_k=i_{k-1}=i_{k+1}$ and $c_{i_{k},i_{k+1}}=-1$.

\end{defi}

We call a total ordering $\le$ on $\Delta^+$ \emph{convex} if for $\beta_1,\beta_2,\beta_1+\beta_2\in \Delta^+$ either $\beta_1 \le \beta_1+\beta_2 \le \beta_2$ holds or $\beta_2 \le \beta_1+\beta_2 \le \beta_1$.
By \cite[Theorem p. 662]{P94} the set of total convex ordering is in natural bijection with the set of reduced words. Namely, for a reduced word $\ii=(i_1,\ldots,i_N)\in\W$ the total ordering \begin{equation*}
\alpha_{i_1}<_{\ii} s_{i_1}(\alpha_{i_2}) <_{\ii} \ldots <_{\ii} s_{i_1}\cdots s_{i_{N-1}}(\alpha_{i_N})
\end{equation*}
on $\Delta^+$
is convex and every convex ordering on $\Delta^+$ arises that way. We write $\Delta^+_{\ii}=\{\beta_1,\beta_2,\ldots,\beta_N\}$ for the set of positive roots ordered with respect to the convex ordering $<_{\ii}$ and throughout identify $\Delta_{\ii}^+$ with $[N]$ via
\begin{equation}\label{posident}
\beta_k \mapsto k.
\end{equation}
We make use of the following alternative labeling of $\Delta_{\ii}^+$ throughout.
\begin{defi}\label{levelroot}
For $a\in I$ we write 
$\{\beta_{\ell} \in \Delta_{\ii}^+ \mid i_{\ell}=a\}=\{\beta_{a,1}, \dots, \beta_{a,m_a}\}$ with
$m_a=m_{a,\ii}\in\N$ and $\beta_{a,1} <_{\ii} \dots <_{\ii}\beta_{a,m_a}$.
\end{defi}

The Weyl group $W$ acts on $\mathfrak{h}^*$ via
$$s_a \mu = \mu- \mu(h_a)\alpha_a \quad a\in I, \mu \in \mathfrak{h}^*.$$
For $a\in I$ we denote by $a^*$ the element of $I$ such that 
\begin{equation}\label{astardef}
\w0(\alpha_a)=-\alpha_{a^*}
\end{equation}

\subsection{Simply-connected algebraic groups}

Let $G$ be the simple simply-connected complex algebraic group with Lie algebra $\mathfrak{g}$. Let $T\subset G$ be a maximal torus with Lie algebra $\mathfrak{h}$. For $a\in I$, let $\varphi_a:SL_2 \rightarrow G $ be the embedding of $SL_2$ corresponding to the simple root $\alpha_a$. We embed the Weyl group $W\simeq \text{Norm}_G(T)/T$ of $\mathfrak{g}$ as a set into $\text{Norm}_G(T)$ via
\begin{equation}\label{Wemb}
s_a\mapsto \bar{s}_a:=\varphi_a\left(\begin{matrix} 0 & -1 \\ 1 & 0 \end{matrix}\right)\in \text{Norm}_G(T).
\end{equation}

We denote by $\mathcal{N}$ and $\mathcal{N}^-$ the maximal unipotent subgroup of $G$ generated by $\{\varphi_a\left(\begin{smallmatrix} 1 & t \\ 0 & 1 \end{smallmatrix}\right) \,|\, a\in I, t\in\ C\}$ and $\{\varphi_a\left(\begin{smallmatrix} 1 & 0 \\ t & 1 \end{smallmatrix}\right) \,|\, a\in I, t\in\C \}$ respectively and set $B=T\mathcal{N}$ and $B^{-}=T\mathcal{N}^-$.

\subsection{Tropicalization}\label{sec:trop}
We start by recalling the notion of tropicalization. Let $\mathbb{G}_m$ be the multiplicative group. For a torus  $\mathbb{T}=\mathbb{G}_m^k$ we denote by $[\mathbb{T}]_{trop}=\Hom(\mathbb{G}_m,\mathbb{T})=\mathbb{Z}^k$ its cocharacter lattice. A positive (i.e. subtraction-free) rational map $f$ on $\mathbb{T}$, $f(x)=\frac{\sum_{u\in I} c_u x^u}{\sum_{u\in J} d_u x^u}$ with {$c_u, d_u \in \mathbb{R_{+}}$}, gives rise to a piecewise-linear map 
$$[f]_{trop} : [\mathbb{T}]_{trop} \rightarrow [\mathbb{G}_m]_{trop}=\mathbb{Z}, \quad x\mapsto \min_{u \in I} \left\langle x,u\right\rangle - \min_{u\in J} \left\langle x,u\right\rangle,$$
where $\langle \cdot ,\cdot \rangle$ is the standard inner product of $\mathbb{Z}^k$. We call $[f]_{trop}$ the \emph{tropicalization} of $f$. More generally,  for a positive rational map 
$$f=(f_1, \dots, f_\ell) : \mathbb{G}_m^k \dashedrightarrow \mathbb{G}_m^{\ell}$$
we define its tropicalization as $$[f]_{trop}:=([f_1]_{trop}, \ldots ,[f_{\ell}]_{trop}): [\mathbb{G}_m^k]_{trop} \rightarrow [\mathbb{G}_m^{\ell}]_{trop}.$$
The function $f$ is called a \emph{geometric lift of }$f$. Note that there a several choices of geometric lifts of a piecewise-linear function.

\section{Lusztig's parametrization}
\subsection{Lusztig's parametrization of the canonical basis}
We denote by $U_q^-$ the negative part of the quantized enveloping algebra of $\mathfrak{g}$.
Let $\ii=(i_1,\ldots,i_N)\in \W$ be a reduced word and $\{\beta_1,\ldots,\beta_N\}=\Delta^+_{\ii}$. In \cite{Lu} a PBW-type basis
$$B_{\ii}=\left\{F_{\ii,\beta_1}^{(x_{\beta_1})} \cdots F_{\\i,\beta_N}^{(x_{\beta_N})} \mid (x_{\beta_1},\ldots, x_{\beta_N})\in \mathbb{N}^{\Delta_{\ii}^+}\right\},$$
of $U_q^-$ is defined, where 
\begin{equation*} F_{\ii, \beta_j}=T_{i_1}\cdots T_{i_{j-1}}F_j\end{equation*} 
is given via the braid group action $T_i$ defined in \cite[Section1.3]{Lu2}, $X^{(m)}$ is the $q$-divided power defined by $X^{(m)}:=\frac{X^m}{[m][m-1]\cdots [2]}$ and $[m]:=q^{m-1}+q^{m-2}+\ldots +q^{-m+1}$.

By \cite{Lu} $B_{{\ii}}$ is basis of $U_q^{-}$. The $\mathbb{Z}[q]$-lattice $\mathfrak{L}$ spanned by $B_{{\bf i}}$ is independent of the choice of reduced expression $\ii$, as is the induced basis $B:=\pi (B_{\ii})$ of $\mathfrak{L}/q\mathfrak{L}$, where $\pi: \mathfrak{L}\rightarrow \mathfrak{L}/ q\mathfrak{L}$ is the canonical projection. There exists a unique basis ${\bf B}$ of $\mathfrak{L}$ whose image under $\pi$ is $B$ and which is stable under the $\mathbb{Q}$-algebra automorphism preserving the generators of $U_q^-$ and sending $q$ to $q^{-1}$. We call ${\bf B}$ the \emph{canonical basis} of $U_q^-$.

\begin{defi} For $\ii\in \W$ and $x=(x_{1}, \ldots, x_{N})\in \mathbb{N}^{N}$, we denote the element $F_{\ii,\beta_1}^{(x_{1})} \cdots F_{\\i,\beta_N}^{(x_{N})}$ by $F^{x}$ and call $x$ its \emph{$\ii$-Lusztig datum}.
Using identification \eqref{posident} we write 
 $$\lci=\mathbb{N}^{\Delta^+_{\ii}}=\mathbb{N}^N$$
for the cone of all $\ii$-Lusztig data. We call $\lci$ the \emph{cone of Lusztig's parametrization} of the canonical basis.
\end{defi}

Lusztig's canonical basis has favorable properties. In particular it projects to a basis of every irreducible finite dimensional $U_q^{-}$-rep\-re\-sen\-ta\-tion. By specializing $q=1$ one obtains a canonical basis for every irreducible finite-dimensional $G$-representation. We therefore obtain a canonical basis of the ring of regular function $H^0(\GN,\mathcal{O}_{\GN})$ which by Lemma \ref{facts} is parametrized by a graded version of $\mathbb{N}^{\Delta^+_{\ii}}$ as defined in Section \ref{section:gradlus}.

\subsection{Transition maps and geometric lifting}
Using the identification \eqref{posident} we associate to the cone $\lci$ of Lusztig's parametrization the torus
\begin{equation*}
\TTil=\mathbb{G}_m^{\Delta^+_{\ii}}=\mathbb{G}_m^{N}.
\end{equation*}
Following \cite[42.2.6]{L93} we introduce transition maps.
\begin{defi}\label{eq:ltrans} We specify $\trl_{\jj}^{\ii}:\TTil  \rightarrow \TTilp$ as follows. If $\jj\in \W$ is obtained from $\ii\in \W$ by a $3$-move at position $k$ we set $y=\trl_{\jj}^{\ii}$ with 
$$
y = \left(x_1, \dots, x_{k-2}, \frac{x_{k}x_{k+1}}{x_{k-1} + x_{k+1}}, x_{k-1} + x_{k+1}, \frac{x_{k-1}x_{k}}{x_{k-1} + x_{k+1}}, x_{k+2}, \dots, x_N \right).
$$
If $\jj\in \W$ is obtained from $\ii\in \W$ by a $2$-move at position $k$ we set
$$\trl_{\jj}^{\ii} \left(x_1, \dots, x_N \right) = \left( x_1, \dots, x_{k-1}, x_{k+1}, x_k, x_{k+2}, \dots, x_N\right).
$$
For arbitrary $\ii, \jj\in \W$ we define $\trl_{\jj}^{\ii}:\TTil \rightarrow \TTilp$ as the composition of the transition maps  corresponding to a sequence of $2-$ and $3-$moves transforming $\ii$ into $\jj$.
\end{defi}

Using Definition \ref{levelroot} and the identification \eqref{posident} we define for $a\in I$ and $\ii\in\W$ the positive regular map $\luslpotia$ on $\TTil$ by
\begin{equation*}
\luslpotia (x) = \displaystyle\sum_{r=0}^{m_a} x_{a,r}.
\end{equation*}
Recalling from Section \ref{sec:trop} that $\tro{\TTil}=\mathbb{Z}^{N}$ we obtain that Lusztig's parametrizations $\lci\subset \Z^N$ are cut out by $\tro{\luslpotia}$: 
\begin{lem}\label{lustigprep} For reduced words $\ii, \jj \in \W$ we have:
\begin{enumerate}
\item $\luslpotja = \luslpotia \circ \trl^{\jj}_{\ii}.$
\item $\lci=\{x \in \tro{\TTil} \mid \forall a \in I: \tro{\luslpotia} (x) \ge 0 \}.$
\end{enumerate}
\end{lem}
\begin{proof} Statement (1) is a straightforward computation and Statement (2) follows directly from the definition.
\end{proof}
We emphasize that Lemma \ref{lustigprep} is simply a reformulation, adapted to our setup, of well-known facts about Lusztig's parametrizations obtained in \cite{L93}, \cite{BZ2}.

\subsection{Lusztig's graded parametrization}\label{section:gradlus}
In this section we provide a geometric lift of the defining inequalities of the graded cone of Lusztig's parametrization of the canonical basis of the ring of regular function $H^0(\GN,\mathcal{O}_{\GN})$. For this we introduce functions $\lusrpotia$ on
\begin{equation*}
\gTTil:= {\mathbb{G}^I_m} \times \TTil:
\end{equation*}
\begin{defi}\label{kappa} 
Using \eqref{astardef} we denote by $\{\lusrpotia\}_{a\in [n]}$ the positive rational functions on $\gTTil$ satisfying:
\begin{enumerate}
\item For $(\lambda, x)\in \gTTil$ one has $\lusrpot{\ii}{i_N}(\lambda,x)=\lambda_{i_N^*}{x_N}^{-1}$.
\item For $\ii, \jj\in \W$ one has 
$\lusrpotja=\lusrpotia \circ (\text{id} \times \trl^{\jj}_{\ii})$ with $\trl^{\jj}_{\ii}$ as in \eqref{eq:ltrans}.
\end{enumerate}
\end{defi}
We introduce the analogue of $\luslpotia$ for $\Lg$ as follows:
\begin{defi}\label{kappavee} 
For $\ii\in\W$ and $a\in I$ we set $\luslpotiald:=\luslpotia$ and define the positive rational functions $\{\lusrpotia\}_{a\in [n]}$ on $\gTTil$ by requiring:  
	\begin{enumerate}
		\item For $(\lambda, x)\in \gTTil$ one has $\lusrpotld{\ii}{i_N}(\lambda,x)=\prod_{b\in I}\lambda_{b^*}^{c_{i_N,b}}{x_N}^{-1}$.
		\item For $\ii, \jj\in \W$ one has 
		$\lusrpotjald=\lusrpotiald \circ (\text{id} \times \trl^{\jj}_{\ii})$ with $\trl^{\jj}_{\ii}$ as in \eqref{eq:ltrans}.
	\end{enumerate}
\end{defi}
\begin{rem} In Corollary \ref{kappareg} we show that $\lusrpotia$ and $\lusrpotiald$ are regular.
\end{rem}

The functions $\tro{\luslpotia}$ cut out \emph{Lusztig's graded parametrization}
\begin{equation}\label{eq:wLusz}
\glci=\left\{(\lambda,x) \in \tro{\gTTil} \,\middle\mid\,  \forall a \in -[n] \cup [n] \, : \, \tro{\luslpotia} (\lambda,x) \geq 0 \right\}.
\end{equation}
Similarly, we define Lusztig's graded parametrization associated to $\Lg$ as
\begin{equation*}
\glcild:=\left\{(\lambda,x)\in [\gr\TTil]_{trop} \,\middle\mid\,  \forall a \in -[n] \cup [n] \, : \, \tro{\luslpotiald} (\lambda,x) \geq 0 \right\}.
\end{equation*}
We have:
\begin{prop}\label{facts} \begin{enumerate}
 \item For $\ii,\jj \in \W$ one has 
$$\glcj=\tro{\text{id}\times\trl^{\ii}_{\jj}} \glci.$$
 \item For $\lambda\in \mathbb{N}^I$ and $\ii\in \W$ the set
 $$\left\{x\in \tro{\TTil} \mid \left(\lambda,x\right) \in \glci\right\}$$ 
 parametrizes Lusztig's canonical basis of the irreducible representation $V(\sum_{a\in I}\lambda_a\omega_a)$.
 \end{enumerate}
\end{prop}
{Before giving the proof, lets us recall that by \cite{Lu} (see also \cite[Proposition 3.6. (i)]{BZ2}) $\lci$ has a crystal structure isomorphic to $B(\infty)$ in the sense of \cite{Ka}. We denote the Kashiwara involution defined in \cite{Ka} by $*$ and define 
			$$\varepsilon^*_{a}(x)=\displaystyle\max_{k\in \mathbb{N}}\{\tilde e_a ^kx^*\in \lci\},
		$$
		where $\tilde e_a$ is the Kashiwara crystal operator corresponding to the simple root $\alpha_a$.
\begin{proof}[Proof of Proposition \ref{facts}]
			Statement (1) is a direct consequence of Lemma \ref{lustigprep}, the definition of the cone of Lusztig's graded parametrization given in \eqref{eq:wLusz} and Definition \ref{kappa}.
			
			For statement (2) note that by \cite[Proposition 8.2]{Ka}, Lusztig's canonical basis of the irreducible representation $V(\sum_{a\in I}\lambda_a\omega_a)$ is para\-metrized by $\{x\in \lci \mid \forall a \in I \,:\, \varepsilon^*_a(x)\le \lambda_a\}$. It thus suffices to show that for $a\in I$
	\begin{equation}\label{grlusparats}
	\varepsilon^*_{a^*}(x)\le \lambda_{a^*} \Leftrightarrow 
\trod{\lusrpotia} (x)\ge 0.
\end{equation}
		For $\ii=(i_1,\ldots,i_N)\in \W$, we define $\ii^{\starop}:=(i^*_N, \ldots, i^*_1)$. By \cite[Proposition 3.3 (iii)]{BZ2} (see also \cite{Lu}) we have for $x=(x_1, \ldots,x_N)\in \lci$ that \begin{equation}\label{Kashinv}x^*=\trod{\trl_{\ii}^{\ii^{\starop}}} \left(x_N,\ldots,x_1 \right).
	\end{equation} 
	By \eqref{Kashinv} and \cite[Proposition 6.1 (i)]{BZ2} we obtain
	\begin{equation}\label{epsex}
	\varepsilon^*_{i_N^*}(x)=\varepsilon_{i_N^*}\left(x^*\right)=\varepsilon_{i_N^*}\left( \trod{ \trl^{\ii}_{\ii^{\starop}}} x^*\right)=x_N.
	\end{equation}
	From \eqref{epsex} and Definition \ref{kappa} we deduce \eqref{grlusparats}.
\end{proof}
\begin{rem} 
In general there is no closed explicit description of $\tro{\lusrpotia}$.
For arbitrary reduced words in type $A$ the explicit form of $\tro{\lusrpotia}$  is obtained in \cite[Theorem 2.16]{GKS} by combinatorial means.	
For two special classes of reduced words it is obtained in \cite{Rei} and in \cite{SST}: In \cite[Proposition 7.4]{Rei} reduced words adapted to the Dynkin quiver $Q=(I, A)$ of $\mathfrak{g}$ satisfying a certain homological condition are treated. By \cite[Corollary 3.23]{Schu} these are precisely the reduced words $\ii$ adapted to $Q$ such that $\omega_a$ spans a minuscule $\mathfrak{g}$-representation for every sink $a\in I$. For a different class of reduced words satisfying a combinatorial condition called "simply-braided for $a\in I$"  (see \cite[Definition 4.1]{SST}), the function $\tro{\lusrpotia}$ can explicitly be obtained from the "bracketing rules" in \cite[Theorem 4.5]{SST}.
\end{rem}}

\section{string parametrization}
\subsection{String parametrization of the canonical basis}
Let $B(\infty)$ be the crystal of $U_q^-$ in the sense of \cite{Ka}. Recall that the \emph{string parametrization} of the canonical basis corresponding to the reduced word $\mathbf{i}=(i_1,i_2,\ldots,i_N)\in\W$ is given by the set of $\ii$-string data of the elements in $B(\infty)$. Here the \emph{$\ii$-string datum} $\stp_{\ii}(b)\in \mathbb{N}^{N}$ of $b\in B(\infty)$ is determined inductively by
\begin{align*}
x_1 &= \displaystyle\max \{ k \in \N \mid \tilde{e}_{i_1}^k b \in B(\infty)\}, \\
x_2 &= \displaystyle\max \{k \in \N \mid \tilde{e}_{i_2}^{k}\tilde{e}_{i_1}^{x_1} b \in B(\infty)\}, \\
& \vdots \\
x_N &= \displaystyle\max \{k\in \N \mid \tilde{e}_{i_N}^k \tilde{e}_{i_{N-1}}^{x_{N-1}}\cdots \tilde{e}_{i_1}^{x_1} b \in B(\infty) \}.
\end{align*}
Following \cite{BZ2,Lit} we call 
$$\sci:=\{ \stp_{\mathbf i} (b) \mid b\in  B(\infty) \}\subset \N^{N}$$ 
the \emph{string cone associated to $\mathbf i$}.

\subsection{Transition maps and geometric lifting}
Using the identification \eqref{posident} we associate to the string cone $\sci$ the torus
\begin{equation*}
\TTis=\mathbb{G}_m^{\Delta^+_{\ii}}=\mathbb{G}_m^N.
\end{equation*}
Following \cite{BZ2} we further introduce positive rational functions $\trs^{\ii}_{\jj} : \TTis \dashedrightarrow \TTjs$ such that the tropicalization $\tro{\trs^{\ii}_{\jj}}$ gives the transition map between the string cones associated to reduced words $\ii, \jj\in \W$:
\begin{defi} We specify $\trs_{\jj}^{\ii}:\TTis  \rightarrow \TTjs$ as follows. If $\jj\in \W$ is obtained from $\ii\in \W$ by a $3$-move at position $k$ we set $y=\trl_{\jj}^{\ii}$ with 
	$$
	y = \left(x_1, \dots, x_{k-2}, \frac{x_{k}x_{k+1}}{x_{k-1}x_{k+1}+x_k}, x_{k-1}x_{k+1}, \frac{x_{k+1}x_{k-1}+x_k}{x_{k+1}}, x_{k+2}, \dots, x_N \right).
	$$
	If $\jj\in \W$ is obtained from $\ii\in \W$ by a $2$-move at position $k$ we set
	$$\trs_{\jj}^{\ii} \left(x_1, \dots, x_N \right) = \left( x_1, \dots, x_{k-1}, x_{k+1}, x_k, x_{k+2}, \dots, x_N\right).
	$$
	For arbitrary $\ii, \jj\in \W$ we define $\trs_{\jj}^{\ii}:\TTis \rightarrow \TTjs$ as the composition of the transition maps corresponding to a sequence of $2-$ and $3-$moves transforming $\ii$ into $\jj$.
\end{defi}
Recall that $[\TTis]_{trop}=\mathbb{Z}^{N}$. By \cite{BZ2, Lit}, we have on $B(\infty)$
\begin{equation*}
\stp_{\jj}=\tro{\trs_{\jj}^{\ii}}\circ\stp_{\ii}.
\end{equation*}

In the remainder of this subsection we introduce certain positive functions $\stlpotia$ on $\TTis$ and show that the string cone $\sci\subset \tro{\TTis}=\Z^N$  is cut out by the functions $\tro{\stlpotia}$. 
\begin{defi}
We denote by $\{\stlpotia\}_{a\in I}$ the positive rational functions on $\TTis$ satisfying:
\begin{enumerate}
	\item For $(\lambda, x)\in \TTis$ one has $\stlpot{\ii}{i_N}(\lambda,x)={x_N}$.
	\item For $\ii, \jj\in \W$ one has 
	$\stlpotja=\stlpotia \circ \trs^{\jj}_{\ii}$ with $\trs^{\jj}_{\ii}$ as in \eqref{eq:ltrans}.
\end{enumerate}	
\end{defi}
\begin{rem} We show in Corollary \ref{stringreg} that $\stlpotia$ is regular.
\end{rem}
\begin{rem} By Theorem \ref{lempot} and Theorem \ref{string2} the function $\stlpotia$ is closely related to the function $\lusrpotia$ given in Definition \ref{kappa}.
\end{rem}
\begin{prop}\label{stringpos} For $\ii \in \W$ we have
	\begin{equation}\label{eq:strinpos}
	\sci=\left\{x \in [\TTis]_{trop} \,\middle\mid\, \tro{\stlpotia}(x)\ge 0 \text{ for all }a\in I,\, \ii\in \W\right\}.
	\end{equation}
\end{prop}

\begin{rem} In general there is no closed explicit description of the function $\tro{\stlpotia}$. Explicit inequalities for the string cone $\sci$ are obtained in \cite{Lit} for a special class of reduced words and in \cite{GP} for all reduced words in type $A$ (also in \cite{BZ2} for arbitrary reduced words but in a less explicit form). In \cite{GKS} we show that the functions $\tro{\stlpotia}$ recovers the string cone inequalities from \cite{GP}.
\end{rem}
Before proving Proposition \ref{stringpos} we recall from \cite{Ka2, NZ} that $\tro{\TTis}$ has the structure of a free crystal in the sense of \cite{DKK} given as follows. For $x=(x_1, \dots, x_N)\in \Z^N=\tro{\TTis}$ and $a\in I$ we set 
\begin{align} \notag
\nu^k (x)&:= x_k +\sum_{\ell=k+1}^N c_{k, \ell} x_\ell,\\ \label{eq:starstring}
\varepsilon_a^* (x) &:= \max\{ \nu^k (x) \mid k\in [N],\, i_k=a\},\\ \notag
f_a^* (x) &:= (x_{\ell} + \delta_{\ell, k(x)})_{\ell\in [N]}.
\end{align}
where $k(x)\in[N]$ is the smallest $k$ with $i_k=a$ and $\nu^{k}(x)=\varepsilon_a^*(x)$. The maps $f_a^*$ and $\varepsilon_a^*$ satisfy
\begin{align}\label{freec1}
\varepsilon_a^* &= \varepsilon_a^* \circ \tro{\trs^{\ii}_{\jj}},\\\label{freec2}
\tro{\trs^{\ii}_{\jj}}  \circ f_a^* &= f_a^* \circ \tro{\trs^{\ii}_{\jj}}.
\end{align}
By \cite{NZ} we have for $a\in I$
\begin{equation}\label{stringstab}
f_a^* \sci \subset \sci.
\end{equation}
\begin{proof}[Proof of Proposition \ref{stringpos}]
	Denoting the set on the right hand side of \eqref{eq:strinpos} by $\mathcal{P}_{\ii}$ we have have $\sci\subset \mathcal{P}_{\ii}$. Furthermore, for $\jj\in\W$ and $x\in\mathcal{P}_{\jj}$ one computes $\varepsilon_{j_N}^* (x)\geq 0$ and
	$$
	\varepsilon^*_{j_N}(x) >0 \Rightarrow (f^*_{j_N})^{-1} (x) \in \mathcal{P}_{\jj}.
	$$
	Thus, by \eqref{freec1} and \eqref{freec2} for $a\in I$ 
	\begin{align}\label{posstab1}
	\varepsilon_a^* (\mathcal{P}_{\ii})&\geq 0,\\\label{posstab2}
	(f_a^*)^{-1} \{x\in\mathcal{P}_{\ii} \mid \varepsilon_a^* (x) >0 \} &\subset \mathcal{P}_{\ii}. \end{align}
	Let $x\in\mathcal{P}_{\ii} \backslash \sci$ minimize $\sum_{a\in I} \varepsilon_a^*(x)$ on $\mathcal{P}_{\ii}\backslash \sci$. We show
	\begin{equation}
	\label{highweight}
	\forall a \in I \,:\, \varepsilon_a^* (x)=0
	\end{equation}
	as follows.
	By \eqref{posstab1} we have for $a\in I$ that $\varepsilon_a^* (x)\geq 0$.
	If $\varepsilon_a^*(x)>0$ then by \eqref{posstab2} we obtain $y:=(f_a^*)^{-1} (x) \in \mathcal{P}_{\ii}$. Since $\varepsilon_a^* (y) = \varepsilon_a^*(x) -1$ we conclude from the minimality assumption on $x$ that $y\in\sci$. Using \eqref{stringstab} we obtain the contradiction $x=f_a^* (y)\in\sci$. Thus \eqref{highweight} holds.
	
	We conclude the proof by deducing $x=0$ from \eqref{highweight} as follows. Assuming $x\neq 0$ we choose $k\in[N]$ and $\jj\in\W$ with	 $(\tro{\trs^{\ii}_{\jj}} (x))_k \neq 0$ and 
	\begin{equation}\label{positivity}
	\forall \jj'\in\W, k'>k \, : \,  (\tro{\trs^{\ii}_{\jj}} (x))_{k'} = 0.
	\end{equation}
	By \eqref{highweight} we have $\nu^{k} (\tro{\trs^{\ii}_{\jj}} (x)) \leq 0$. Thus, by \eqref{positivity} 
	\begin{equation}\label{kcorneg}
	(\tro{\trs^{\ii}_{\jj}} (x))_k = \nu^{k} (\tro{\trs^{\ii}_{\jj}} (x)) < 0.
	\end{equation}
	From $x\in\mathcal{P}_{\ii}$ we conclude $k<N$. Thus, there exists $\jj', \jj'' \in\W$ with 
	\begin{equation}\label{kcorsame}(\tro{\trs^{\ii}_{\jj'}} (x))_{k} = (\tro{\trs^{\ii}_{\jj}} (x))_{k}
	\end{equation}
	such that $\jj''$ is obtained from $\jj'$ either by a $2$-move at position $\ell'=k$ or by a $3$-move at position $\ell'\in\{k, k+1\}$. By \eqref{positivity}, \eqref{kcorneg} and \eqref{kcorsame} we have $(\tro{\trs^{\ii}_{\jj''}} (x))_{\ell'+1}<0$, which contradicts \eqref{positivity}.
\end{proof}
\begin{rem} In the special case that $\mathfrak{g}$ is of type $A$, a proof of the equality $\mathcal{S}_{\ii}=\mathcal{P}_{\ii}$ was obtained in \cite{GP} using explicit defining inequalities for $\sci$ derived in op. cit.
\end{rem}

\subsection{Graded string cones}
Following \cite{Lit}, we define the \emph{graded string cone} 
\begin{equation*}
\gsci:=\left\{(\lambda,x)\in \mathbb{Z}^I\times \mathcal{S}_{\ii} \mid \lambda_{a}  \ge \displaystyle\max_{i_k=a} \left\{x_k + \displaystyle\sum_{\ell=k+1}^{N}c_{i_{\ell},i_{k}}x_{\ell}\right\} \text{ for all $a\in I$}\right\},
\end{equation*}
which parametrizes a basis of $H^0(\GN,\mathcal{O}_{\GN})$ by \cite[Proposition 1.5]{Lit}. 
\begin{rem}
By \cite{Ka,NZ} the string cone $\sci$ has a crystal structure isomorphic to $B(\infty)$ with $\varepsilon^*_a(x)$ given by \eqref{eq:starstring}.
Furthermore, by \cite[Proposition 8.2]{Ka} Lusztig's canonical basis of the irreducible representation $V(\sum_{a\in I}\lambda_a\omega_a)$ is para\-metrized by the set of $x\in \mathcal{S}_{\ii}$ such that $\varepsilon^*_a(x)\le \lambda_a$ for all $a\in I.$ This gives an alternative proof that $\gsci$ parametrizes Lusztig's canonical basis of $H^0(\GN,\mathcal{O}_{\GN})$.
\end{rem}

In the following we introduce positive functions on 
\begin{equation*}
\gTTis:= {\mathbb{G}^I_m} \times \TTis,
\end{equation*}
whose tropicalization cut out $\sci$:
\begin{defi}\label{nu} For $a\in I$ we specify on $\gTTis$ the positive function 
	$$\strpotia (\lambda,x)=\lambda_a\displaystyle\sum_{\substack{k\in [N] \\ i_k=a}}x^{-1}_k\displaystyle\prod_{\ell=k+1}^N x^{-c_{i_{\ell},i_k}}_{\ell}.$$
\end{defi}
\begin{prop}\label{wstringpos} For reduced words $\ii, \jj \in \W$ we have
	\begin{enumerate}
		\item $\strpotja = \strpotia \circ \trs^{\ii}_{\jj},$
\item $\gsci=\left\{(\lambda,x) \in \tro{\gr\TTis} \,\middle\mid\,  \forall a \in -[n] \cup [n] \, : \, \tro{\stlpotia} (\lambda,x) \geq 0 \right\}.$
	\end{enumerate}
\end{prop}
\begin{proof}
	Statement (1) is a straightforward computation. 
	Statement (2) follows from Proposition \ref{stringpos}.
\end{proof}
We introduce the analogues of $\stlpotia$ and $\gsci$ for $^{L}\mathfrak{g}$ as follows.
\begin{defi}\label{nulld}
	For $a\in I$ we specify on $\gTTis$ the positive functions
	$\stlpotiald:=\stlpotia$ and 
	$$\strpotiald (\lambda,x)=\left(\prod_{b\in I}\lambda^{c_{a,b}}_b\right) \displaystyle\sum_{\substack{k\in [N]\\ i_k=a}}x^{-1}_k\displaystyle\prod_{\ell=k+1}^N x^{-c_{i_{\ell},i_k}}_{\ell}.$$
\end{defi}
\begin{defi} We introduce the graded string cone of $^{L}\mathfrak{g}$ as
	$$\gscild:=\left\{(\lambda,x) \in \tro{\gr\TTis} \,\middle\mid\,  \forall a \in -[n] \cup [n] \, : \, \tro{\stlpotiald} (\lambda,x) \geq 0 \right\}.$$
\end{defi}

\section{The cluster spaces of the base affine space}
\subsection{Generalized minors and the open double Bruhat cell}\label{minors}
In the following we identify the weight lattice of $\mathfrak{g}$ with the group of multiplicative characters on $T$. For a dominant weight $\lambda: T \rightarrow \mathbb{G}_m$, we define the \emph{principal minor} $\Delta_{\lambda}:G \rightarrow \mathbb{A}^1$ to be the function defined on the open subset $\mathcal{N}^-T\mathcal{N} \subset G$ by
$$\Delta_{\lambda}(u^-tu^+):= \lambda(t) \quad u^-\in \mathcal{N}^-, t\in T, u^+\in \mathcal{N}.$$
Let $\gamma,\delta$ be extremal weights such that $\gamma=w_1\lambda$, $\delta=w_2\lambda$ for some $w_1,w_2\in W$, $\lambda\in P^+$. Recall the embedding of sets \eqref{Wemb} of $W$ into $\text{Norm}_G(T).$
The \emph{generalized minor} associated to $\gamma$ and $\delta$ is 
$$\Delta_{\gamma,\delta}(g):=\Delta_{\lambda}(\overline{w}_1^{-1}g \overline{w}_2), \quad g\in G.$$

The base affine space $\GN$ is the partial compactification of the \emph{open double Bruhat cell} 
$$G^{\w0,e} := B\w0 B \cap B_-$$
obtained by allowing the generalized minors $\Delta_{\omega_a,\omega_a}$ and $\Delta_{\w0\omega_a,\omega_a}$ to vanish (see \cite[Section 2.6]{BFZ2}).

\subsection{The cluster ensemble of the open double Bruhat cell}

\subsubsection{Cluster seeds}
Following  Fomin-Zelevinsky, Berenstein-Fomin-Zelevinsky and Fock-Goncharov \cite{FZ, BFZ2,FZ2, FG} we recall the definition of the cluster ensemble associated to $G^{\w0,e}$. We start by introducing the notion of a seed.
\begin{defi} A seed is a datum  $\SE=(\Lambda, \langle \cdot, \cdot \rangle, \{e_k\}_{k\in M}, M_0)$, where
		\begin{itemize}
		\item[(i)] $\Lambda$ is a lattice,
		\item[(ii)] $\langle \cdot, \cdot \rangle$ is a skew-symmetric $\mathbb{Z}$-valued bilinear form on $\Lambda$,
   	\item[(iii)] $M_0\subset M$ are finite set,
		\item[(iv)] $\{{e_k}\}_{k\in M}$ is a basis of $\Lambda$.
	\end{itemize}
We associate to a seed $\SE=(\Lambda, \langle \cdot, \cdot \rangle, \{e_k\}_{k\in M}, M_0)$ a quiver $\Gamma_{\SE}$  as follows.
	The vertices $\{v_k\}_{k\in M}$ of $\Gamma_{\SE}$ are indexed by $M$. If $\left<e_k,e_{\ell}\right> >0$ then there are $\left<e_k,e_{\ell}\right>$ arrows with source $v_k$ and target $v_{\ell}$ in $\Gamma_{\SE}$. A vertex $v_k$ is called \emph{frozen} if $k\in M_0$ and \emph{mutable} if $k\in M\setminus M_0$.
\end{defi}

Let $\SE=(\Lambda, \langle \cdot, \cdot \rangle, \{e_k\}_{k\in M}, M_0)$ be a seed. For each $k\in M \setminus M_0$ we define the seed $\mu_k(\Gamma_{\Sigma})=(\Lambda, \langle \cdot, \cdot \rangle, \{e'_k\}_{k\in M}, M_0)$, called the \emph{mutation of $\Sigma$ at $k$}, by setting
\begin{equation*}
e'_j=\begin{cases} e_j+\max\{0, \left<e_j,e_k\right>\}e_k & \text{ if }j\ne k \\
-e_k & \text{ if }j=k.
\end{cases}
\end{equation*}
The quiver $\mu_k(\Gamma_{\Sigma}):=\Gamma_{\mu_k \SE}$, called the \emph{mutation of $\Gamma_{\SE}$ at $k$}, is obtained as follows. The vertices and frozen vertices of $\Gamma_{\SE}$ and $\mu_k(\Gamma_{\SE})$ coincide.
Furthermore $\mu_k(\Gamma_{\SE})$ has the same arrows as $\Gamma_{\SE}$, except:
\begin{itemize}
	\item[(i)] All arrows of $\Gamma_{\Sigma}$ with source or target $v_k$ get replaced in $\mu_k(\Gamma_{\Sigma})$ by the reversed arrow.
	\item[(ii)] For every pair of arrows $(h_1, h_2)\in \Gamma_{\Sigma}\times \Gamma_{\Sigma}$ with 
	$$v_k=\text{target of $h_1$}=\text{source of $h_2$}$$
	we add to $\mu_k(\Gamma_{\Sigma})$ an arrow from the source of $h_1$ to the target of $h_2$.
	\item[(iii)] If a $2$-cycles was obtained during (i) and (ii), the arrows of this $2$-cycle get canceled in $\mu_k(\Gamma_{\Sigma})$.
	\item[(iv)] Finally we erase all arrows between frozen vertices.
\end{itemize}

To a seed $\SE=(\Lambda, \langle \cdot, \cdot \rangle, \{e_k\}_{k\in M}, M_0)$ we assign the $\mathcal{A}$- and $\mathcal{X}$-cluster torus
$$\TTsea:=\text{Spec}\mathbb{Z}[A_k^{\pm 1} \mid k\in M] \quad \text{and} \quad \TTsex:=\text{Spec}\mathbb{Z}[X_k^{\pm 1} \mid k\in M].$$

We introduce birational $\mathcal{A}$-cluster transformations $\mu_k : \TTse   \dashedrightarrow \TT_{\mu_k \SE}$ and $\mathcal{X}$-cluster transformations ${\dual{\mu}_k} : \TTsex  \dashedrightarrow \TTx_{\mu_k \SE}$:
\begin{align}\label{amu}
 \mu_k^* A_{\ell} &= \begin{cases} 
\displaystyle\prod_{j \,:\, \langle e_j, e_k \rangle>0} \frac{A_j^{\langle e_j, e_k \rangle}}{A_k} + \displaystyle\prod_{j \,:\, \langle e_j, e_k \rangle < 0} \frac{A_j^{-\langle e_j, e_k \rangle}}{A_k} & \text{if $\ell=k$,} \\A_{\ell} &  \text{else,} \end{cases} \\\label{xmu}
\dual{\mu}_k^* X_\ell &= \begin{cases} X_k^{-1}\qquad\quad & \text{if $\ell=k$,} \\ X_\ell (1+X_k^{-\sgn \langle e_k, e_\ell \rangle})^{-\langle e_k,e_\ell \rangle} \qquad \quad\quad\,\,\,&  \text{else.} \end{cases}
\end{align}
For seeds $\SE$ and $\SEp$ obtained by a sequence of mutations we define $\ttsa : \TTse   \dashedrightarrow \TTsep$ and $\ttsx : \TTsex  \dashedrightarrow \TTsexp$ by composition.

\subsubsection{Seeds associated to reduced words}\label{redquiver} 
Following \cite{BFZ2} we associate to every reduced word $\ii\in \W$ a seed $\Sigma_{\ii}$. We throughout identify 
$$
\Sigma_{\ii} \text{ and } \ii,
$$
i.e. we denote the seed $\SE_{\ii}$ also by $\ii$. The quiver $\Gamma_{\ii}$ can be described as follows. We denote the vertices of $\Gamma_{\ii}$ by $\{v_k \mid k\in M \}$, where
$$M:=k\in \{-1,\ldots,-n\}\cup \{1,\ldots,N\}.$$ 
Using Definition \ref{levelroot} we write $v_{\ell}=v_{a,r}$ if $\beta_{\ell}=\beta_{a,r}$ and set $v_{a,0}:=v_{-a}$. The frozen vertices of $\Gamma_{\ii}$ are 
\begin{equation}\label{frlabel}
\{w_{-a}:=v_{a,0} \mid a\in I\} \cup \{w_a:=v_{a,m_a} \mid a\in I\}.
\end{equation}

In order to define the arrows in $\Gamma_{\ii}$ we introduce the following notion.
For $k\in [-n]$ we set $i_k=-k$. For $k\in M$ we denote by $k^+=k^+_{\ii}$ the smallest $\ell \in M$ such that $k<\ell$ and $i_{\ell}=i_k$. If no such $\ell$ exists, we set $k^+=N+1$. For $k\in [N]$, we further let $k^-$ be the largest index $\ell \in M$ with that $\ell<k$ and $i_{\ell}=i_k$. 

There is an edge connecting $v_k$ and $v_{\ell}$ with $k<\ell$ if at least one of the two vertices is mutable and one of the following conditions is satisfied:
\begin{enumerate}
	\item $\ell=k^+$,
	\item $\ell < k^+<\ell^+$, $c_{k,\ell}<0$ and $k,\ell \in [N]$.
\end{enumerate}
Edges of type (1) are called \emph{horizontal} and are directed from $k$ to $\ell$. Edges of type (2) are called \emph{inclined} and are directed from $\ell$ to $k$.
\begin{rem}The quiver $\Gamma_{\ii}$ associated to a reduced word $\ii$ has between any two vertices $v,w$ at most $1$ edge, which we denote by $[v,w]$.
\end{rem}

\begin{ex} Let $\mathfrak{g}=\text{sl}_3(\mathbb{C})$ and $\ii=(1,2,1)$. Then $\Gamma_{\ii}$ looks as follows:
	$$
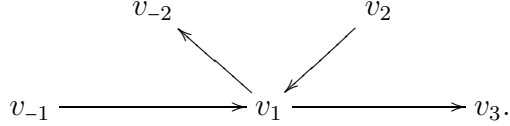
\xymatrix{
		&v_{-2} & & v_2 \ar[ld]\\
		v_{-1} \ar[rr] & & v_1 \ar[rr] \ar[lu] & & v_3.
	}$$
\end{ex}

\begin{lem}\label{braid1} Let $\jj\in \mathcal{W}(\w0)$ be obtained from $\ii\in \mathcal{W}(\w0)$ by a $2$-move or a $3$-move in position $k$. Then the transposition $(k,k+1)$ is an isomorphism of quivers $\Gamma_{\jj}\simeq\mu_k \Gamma_{\ii}.$
\end{lem}
\begin{proof}
	The statement follows from the construction of $\Gamma_{\ii}$ and \cite[Theorem 3.5]{SSVZ}.
\end{proof}

We consider the families $(\TTsea)$ and $(\TTsex)$ of all tori corresponding to seeds $\Sigma$ which can be obtained from a seed $\Sigma=\ii$ corresponding to a reduced word $\ii \in \W$ by a sequence of mutations.
By \cite{BFZ2} the open double Bruhat cell $G^{\w0,e}$ is covered up to codimension $2$ by $(\TTsea)$ via 
\begin{align}\begin{split}\label{atorusem}
G^{\w0,e} &\dashrightarrow \TTsea,\\
g&\mapsto \left(\Delta_{s_{i_1}\cdots s_{i_{k}}\omega_{i_k},\omega_{i_k}} (g)\right)_{k\in M}.
\end{split}
\end{align}
The associated gluing maps are given in \eqref{amu}.
\begin{rem}
We defer from the convention in \cite{BFZ2} as follows. The seed we associate to a reduced word $\ii\in\W$ coincides with the seed associated to $-\ii$ in op. cit. obtained by reversing all arrows.
\end{rem}

The \emph{cluster space $\mathcal{A}$} and the \emph{dual cluster space $\mathcal{X}$} associated to $G^{\w0,e}$ is the scheme obtained by gluing the tori $(\TTsea)$ and $(\TTsex)$ via \eqref{amu} and \eqref{xmu}, respectively. We call the pair $(\mathcal{A},\mathcal{X})$ the cluster ensemble associated $G^{\w0,e}$.

We use the following fact later.
\begin{lem}\label{diaglev} For $a\in I$ and reduced words $\ii,\jj\in \W$ we have 
	$$\prod_{r=1}^{m_{a,\ii}} x_{a,r}= \prod_{r=1}^{m_{a,\jj}} x_{a,r} \circ\ttia.$$
\end{lem}
\begin{proof} 
		Without loss of generality we can assume that $\jj$ is obtained from $\ii$ by a $3$-move at position $\ell$ with $v_{\ell}=v_{a,s}$. The claim then follows since we have for $r\in[m_a^{\jj}]$
	$$
	x_{a,r} \circ \ttix = \begin{cases} 	x_{a,r}  & \text{if $r<s-1$,}\\ 
	x_{a,r} (1+(x_{a,s})^{-1})^{-1}& \text{if $r=s-1$,}
	\\ 	x_{a,r-1} (1+x_{a,s})& \text{if $r=s$,}
	\\ 	x_{a,r-1} & \text{if $r>s$.}
	\end{cases}
	$$
\end{proof}

\section{Gross-Hacking-Keel-Kontsevich potential and Berenstein-Kazhdan decoration function}\
\subsection{Gross-Hacking-Keel-Kontsevich potential}\label{GHKK}
Recall from Section \ref{minors} that the base affine space $\GN$ is the partial compactification of the open double Bruhat cell $G^{\w0,e}$: $$\GN=G^{\w0,e} \cup \displaystyle\bigcup_{\pm a\in [n]} D_a,$$
where $D_a$ is the divisor given by the vanishing locus of the functions $\Delta_{\omega_a,\omega_a}$ for $a<0$ and $\Delta_{\w0\omega_a,\omega_a}$ for $a>0$, corresponding to the frozen vertices of $\Gamma_{\ii}$ by \eqref{atorusem}.

In \cite{GHKK} a Landau-Ginzburg potential $W$ on the dual cluster space $\mathcal{X}$ associated to $\GN$ is defined as the sum $W=\sum_{\pm a\in [n]}W_a$ of certain global monomials $W_a$ attached to the divisors $D_a$. We are interested in $\restr{W}{\TTix}$ since the cone
\begin{equation*}%\label{cconedef}
\xcs:=\left\{x \in \tro{\TTsex}  \,\middle\mid\,  \tro{\restr{W}{\TTsex}} (x)\ge 0\right\}
\end{equation*}
cut out by the tropcialization of $\restr{W}{\TTix}$, up to a natural conjecture (see Remark \ref{GHKKrem}), parametrizes a canonical basis for the ring of regular functions on the partial compactification $\GN$ of $G^{\w0,e}$.

Using \eqref{frlabel} we have the following definition of $W_a$, which in \cite[Corollary 9.17]{GHKK} is shown to be well-defined.
\begin{defi}
If there is no arrow in $\Gamma_{\SE}$ from the frozen vertex $w_a$ to a mutable vertex we call \emph{$\SE$ optimized for $w_a$} and have $\restr{W_a}{\TTsex}=(\xse{w_a})^{-1}$.
\end{defi}

For certain toric charts we have a closed explicit description of $\restr{W_a}{\TTix}$:
\begin{prop}\label{locpotGHKK1}
Every frozen vertex $w_a$ has an optimized seed. Furthermore, for $a\in I$ and $\ii=(i_1,\ldots,i_{N})\in \W$ we have
\begin{align}
\label{eq:rightGHKKpot}
\restr{W_{i_N}}{\TTix} (x) &=\xv{N}^{-1},\\\label{eq:leftGHKKpot}
\restr{W_{-a}}{\TTix}(x)&=\sum_{k=0}^{m_a-1}\prod_{\ell=0}^{k} \xvc{a}{\ell}^{-1}.
\end{align}
\end{prop}
\begin{proof}
 By definition the quiver $\Sigma_{\ii}$ is optimized for the frozen vertex $v_{i_N}$ and \eqref{eq:rightGHKKpot} follows. It remains to show \eqref{eq:leftGHKKpot} and that for $a\in I$ the vertex $w_{-a}$ has an optimized seed.
 
Let $\ii\in \W$, $j\in [m_a-1]$ and let $\Gamma_{\ii}^{(j)}$ be the resulting quiver after applying the sequence of mutations at the vertices $v_{a,1},v_{a,2},\ldots,v_{a,j}$ to $\Gamma_{\ii}$.

Between two vertices there is at most one arrow in $\Gamma_{\ii}^{(j)}$. Furthermore, we have for $j\geq 2$
\begin{equation}\label{locopt}
[v_{0,a},v_{a,j+1}],\, [v_{a,j+1}, v_{a,j}]\in\Gamma_{\ii}^{(j)}.
\end{equation}
We prove that $[v_{0,a},v_{a,j+1}]$ is the only arrow in $\Gamma_{\ii}^{(j)}$ with source $v_{0,a}$. 
	
	Let $j$ be minimal such that there exist an arrow $[v_{0,a},w]$ in $\Gamma_{\ii}^{(j)}$ with $w\ne v_{a,j+1}$ and $w$ mutable. Then $[v_{a,j-1},w]$ is an arrow of $\Gamma_{\ii}^{(j-1)}.$ 
	Note that $[v_{a,j-1},w]$ has to be an inclined arrow of $\Gamma_{\ii}$. Since $w$ and  $v_{a,j-1}$ are mutable, there exists an inclined arrow $[w,v_{a,r}]$ in $\Gamma_{\ii}$ with $r<j-1$. This arrow stays unchanged under the sequence of mutations at $v_{a,1},v_{a,2},\ldots,v_{a,r-1}$, creates an arrow $[w, v_{a,r-1+s}]$ in $\Gamma_{\ii}^{(r+s)}$ for $r+s < j$ and cancels the arrow $[v_{0,a},w]$ in $\Gamma_{\ii}^{(j)}$ yielding a contradiction.
  
    Hence the seed $\SE$ corresponding to $\Gamma_{\ii}^{m_a-1}$ is optimized for $v_{a,0}=w_{-a}$. Furthermore, from \eqref{locopt} we recursively compute \eqref{eq:leftGHKKpot}.
\end{proof}
\begin{ex}
	For $\mathfrak{g}=\text{sl}_3(\mathbb{C})$ and $\ii=(1,2,1)$ we have
	\begin{align*}
	\restr{W_{-1}}{\TTix}(x)&=x_{-1}^{-1} + x_{-1}^{-1}x_1^{-1}, &
	\restr{W_{1}}{\TTix}(x)&=x_3^{-1},\\
	\restr{W_{-2}}{\TTix}(x)&=x_{-2}^{-1}, &
	\restr{W_{2}}{\TTix}(x)&=x_2^{-1} + x_1^{-1}x_2^{-1}.
	\end{align*}
\end{ex}

\begin{rem}\label{GHKKrem} In \cite{GHKK} a canonical basis for the ring of regular functions on an $\mathcal{A}$-cluster variety, called theta basis, is constructed under the assumptions given in op. cit. Definition 0.6. This assumption is called the full Fock-Goncharov conjecture and ensures in particular, that the theta basis is naturally identified with the tropical points of the corresponding $\mathcal{X}$-cluster variety. 

In the case of double Bruhat cells, Goodearl and Yakimov  announced in \cite{GY} (see \cite[Example 0.15]{GHKK}) the existence of a maximal green sequence which implies the full Fock-Goncharov conjecture for $G^{e,w_0}$ by \cite[Proposition 0.14]{GHKK}. 

By Proposition \ref{locpotGHKK1} every frozen vertex of $\Gamma_{\ii}$ has an optimized seed. Thus using \cite[Proposition 2.6.]{BFZ2}, \cite[Theorem 0.19, Lemma B.7]{GHKK} and the existence of a maximal green sequence for $G^{e,w_0}$, there exists a theta basis for the partial compactification $\GN$ of $G^{\w0,e}$ parametrized by $\xcs$.
\end{rem}

 \subsection{Berenstein-Kazhdan decorations}\label{BK}
In \cite[Corollary 1.25]{BK2}, Berenstein and Kazhdan introduced as part of the datum of a $G$-decorated geometric crystal the \emph{decoration function} $f^B=\sum_{a\in\pm I} f^B_a$ on $G$, where for $a\in I$
\begin{equation*}
f^B_{-a}(g)= \frac{\Delta_{\w0\omega_a, s_a\omega_a}(g)}{\Delta_{\w0\omega_a,\omega_a}(g)}, \qquad 
f^B_{a}(g) = \frac{\Delta_{\w0s_a\omega_a,\omega_a}(g)}{\Delta_{\w0\omega_a,\omega_a}(g)}.
%\label{def:rBKpot}
\end{equation*}

For certain toric charts we have a closed explicit description of $f_a^B$: 

\begin{prop}\label{apotex}
For $a\in I$ and $\ii=(i_1,\ldots,i_{N})\in \W$ we have
	\begin{align}
	\label{eq:rightBKpot}
	\restr{f^B_{i_N}}{\TTia}(x) &=  \xv{N}^{-1}\xv{N^{-}} ,\\
	\label{eq:leftBKpot}
	\restr{f^B_{-a}}{\TTia}(x)&=\sum_{\substack{k\in [N] \\ i_k=a}} x_{k^-}^{-1} x_k^{-1}
	\prod_{\substack{\ell\in M \\ \ell<k<\ell^+}} x_{\ell}^{-c_{i_{\ell}, a}}.
	\end{align}
\end{prop}
\begin{proof}
Equality \eqref{eq:rightBKpot} follows from \eqref{atorusem}. Equality \eqref{eq:leftBKpot} follows from \eqref{atorusem} and \cite[Equation (5.8)]{BZ2}, which holds on $G^{w_0,e}$.
\end{proof}
\begin{ex}
For $\mathfrak{g}=\text{sl}_3(\mathbb{C})$ and $\ii=(1,2,1)$ we have
\begin{align*}
\restr{f^B_{-1}}{\TTia}(x)&=\frac{x_{-2}}{x_{-1}x_1}+ \frac{x_{2}}{x_1 x_3}, &
\restr{f^B_{1}}{\TTia}(x)&=\frac{x_{1}}{x_3},\\
\restr{f^B_{-2}}{\TTia}(x)&=\frac{x_{1}}{x_{-2} x_2}, &
\restr{f^B_{2}}{\TTia}(x)&={\frac{x_{-1}}{x_1} + \frac{x_{-2}x_3}{x_1 x_2}}.
\end{align*}
\end{ex}

We denote by $\acs$ the cone cut out by the decoration function $f^B$:
$$\acs=\{v \in \tro{{\mathbb{T}_{\Sigma}}} \mid \tro{\restr{f^B}{\mathbb{T}_{\Sigma}}} (v)\ge 0\}.$$

\section{Lusztig parametrization via cluster varieties}\label{lustigsec}
In this section we relate the cone  $\glci$ of Lusztig's graded parametrization to the cone $\xci$ cut out by the tropicalization of the Gross-Hacking-Keel-Kontsevich potential function $W$ introduced in Section \ref{GHKK}. Dually, we relate the cone ${\glcild}$ associated to the Langlands dual Lie algebra $\Lg$ to the cone $\aci$ cut out by the tropicalization of the decoration function $f^B$ due to Berenstein-Kazhdan defined in Section \ref{BK}.

Motivated by the \emph{Chamber Ansatz} due to Berenstein-Fomin-Zelevinsky \cite{BFZ} as well as by \cite[Equation (4.14)]{BZ2} and \cite[Section 5]{GKS} we introduce the following coordinate transformations using the notations of Section \ref{redquiver} and \eqref{astardef}.
First, we set for $k, \ell\in M=-[n] \cup [N]$
\begin{equation}
\label{sym}
\sym{k}{\ell} := \begin{cases} 1&\text{if $k < \ell < k^+$,}\\
-1 &\text{if $\ell=k$ or $\ell= k^+$}\\
0& \text{else.}\end{cases}
\end{equation}
\begin{defi}
We specify $\giotid \in\Hom(\gTTil, \TTix)$ and $\gCAi\in \Hom(\TTia,\gTTil)$:
\begin{align*}
(\giotid (\lambda, x))_{k} &= \begin{cases} x_{k^+}^{-1} &\text{if $k<0$,}\\
{x_k}{x_{k^+}}^{-1} &\text{if $k^+\in [N],$}\\
x_{k}\lambda_{a^*}^{-1} &\text{if $k^+=N+1$,}
\end{cases}\\
\gCAi (x)&=\left(
\left(x_{a^*,m_{a^*}}^{-1}\right)_{a\in I},  \left( 
\prod_{{\ell\in M}} x_{\ell}^{\sym{\ell}{k}}\right)_{k \in [N]} \right).
\end{align*}
\end{defi}
\begin{ex} Let $\mathfrak{g}=\text{sl}_3(\mathbb{C})$ and $\ii=(1,2,1)$. Then we have 
\begin{align*}
		\giotid \left(\lambda_1,\lambda_2,x_1,x_2,x_3\right)&=\left(\frac{1}{x_1}, \frac{1}{x_2},  \frac{x_1}{x_3}, \frac{x_2}{\lambda_1}, \frac{x_3}{\lambda_2}\right),\\
		\gCAi\left(x_{-1},x_{-2},x_1,x_2,x_3\right)&=\left(\frac{1}{x_2},\frac{1}{x_3},\frac{x_{-2}}{x_{-1}x_1}, \frac{x_1}{x_{-2}x_2},\frac{x_2}{x_1 x_3}\right).
\end{align*}
\begin{rem}
In Lemma \ref{unicones} we show that $\tro{\giotid}$ and $\tro{\gCAi}$ 
and consequently also ${\giotid}$ and ${\gCAi}$
are isomorphisms.
\end{rem}
\end{ex}
The families $(\giotid)$ and ($\gCAid$) have the following transformation behaviour.
\begin{lem}\label{diagwlX}
	For $\ii, \jj\in\W$ the following diagrams commutes.
	$$\xymatrix{
		\TTia \ar[dd]^{\trl^{\ii}_{\jj}} \ar[rr]^{\gCAi} & & \gTTil \ar[rr]^{\giotid}\ar[dd]_{\text{id} \times \trl^{\ii}_{\jj}} & &\TTix \ar[dd]^{\trs^{\ii}_{\jj}} \\
		& \\
		\TTiap \ar[rr]^{\gCAj} & & \gTTjl \ar[rr]^{\giotjd} & & \TTixp.}
	$$
\end{lem}
\begin{proof}
	Without loss of generality we can assume that $\jj$ is obtained from $\ii$ by a $3$-move at position $k$ with $i_{k}=a$.
Then by Lemma \ref{braid1}, the tori $\TTixp$ and $\TTiap$ are obtained from $\TTix$ and $\TTia$ by $\mathcal{X}$-cluster and $\mathcal{A}$-cluster transformation at the $4$-valent vertex $k-1$ of $\Gamma_{\ii}$, respectively. The commutativity of the diagram then can be checked by direct computation.
\end{proof}

We relate the GHKK-potential and the BK-decoration function to the functions $\lusrpotia$ and $\lusrpotiald$ introduced in Definition \ref{kappa} and \ref{kappavee}, respectively:
\begin{thm}\label{lempot}  For $a\in \pm I$ and $\ii \in \W$ we have 
	\begin{align}\label{lustigpot=cluster}
\lusrpotia&=	\restr{W_{a}}{\TTix}\circ \giotid,\\
\label{lustigpot=clustera}
\restr{f^B_{a}}{\TTia}&={\lusrpotiald}\circ \gCAi.
	\end{align}
\end{thm}
	\begin{proof}
For $a<0$ equalities \eqref{lustigpot=cluster} and \eqref{lustigpot=clustera} follow directly from Proposition \ref{locpotGHKK1} and Proposition \ref{apotex}, respectively. For $a>0$ equality \eqref{lustigpot=cluster}  follows directly from Lemma \ref{diagwlX} and Proposition \ref{locpotGHKK1}. In order to show \eqref{lustigpot=clustera} we compute using Lemma \ref{diagwlX} and Proposition \ref{apotex} for $a=i_N$:
$$
{\lusrpotiald}\circ \gCAi (x) = (\prod_{b\in I}x_{b,m_b}^{-c_{a,b}})  x_{a,m_a-1}x_{a,m_a}\prod_{
b\neq a} x_{b,m_b}^{c_{a,b}} = \frac{x_{N^-}}{x_N} = \restr{f^B_{a}}{\TTia} (x).
$$
\end{proof}
As a direct corollary of Theorem \ref{lempot} we obtain:
\begin{cor}\label{kappareg} For $a\in I$ the functions $\lusrpotia$ and $\lusrpotiald$ are regular.
\end{cor}
Theorem \ref{lempot} has the following implications for the interplay of parametrizations of canonical bases.
\begin{cor}\label{lusthm} We have the following equality of cones: 
	\begin{align*}
 \xci&=\tro{\giotid} (\glci), \\
\glcild&=\tro{\gCAi} (\aci).
	\end{align*}
\end{cor}
\begin{proof}
The statement follows by tropicalization Theorem \ref{lempot}.
\end{proof}

\section{string parametrization via cluster varieties}
In analogy to Section \ref{lustigsec} we relate in this section the graded string cone $\gr\sci$ to the cone $\xci$ cut out by the tropicalization of the GHKK-potential function $W$.
Dually, we relate the cone ${\gscild}$ associated to the Langlands dual Lie algebra $\Lg$ to the cone $\aci$ cut out by the tropicalization of the BK-decoration function $f^B$.

Motivated by the \emph{Chamber Ansatz} due to Berenstein-Fomin-Zelevinsky \cite{BFZ} as well as by \cite[Equation (4.14)]{BZ2} and \cite[Section 5]{GKS} we introduce the following coordinate transformations using the notations of Section \ref{redquiver}, \eqref{sym} and \eqref{astardef}.
\begin{defi}
	We specify $\gCAid \in\Hom(\gTTis, \TTix)$ and $\gioti\in \Hom(\TTia,\gTTis)$:
	\begin{align*}
	(\gCAid (\lambda, x))_{k} &= \begin{cases} \displaystyle
	\lambda_{-k}^{-1} \prod_{\ell\in[N]} x_{\ell}^{c_{i_\ell, i_{k}}+\{k,\ell\}}
    &\text{if $k<0$,}\\[22pt]\displaystyle
	\prod_{\ell\in [N]} x_{\ell}^{\sym{k}{\ell}}&\text{else,}
	\end{cases}\\
	\gioti x&=\left(\left(
	x_{a,m_{a}}^{-1}\right)_{a\in I}, \left({x_{k}}^{-1}{x_{k^{-}}} \right)_{k	 \in [N]}\right).
	\end{align*}
\end{defi}
\begin{rem}
	In Lemma \ref{unicones} we show that $\tro{\gCAid}$ and $\tro{\gioti}$  
	and consequently also ${\gCAid}$ and ${\gioti}$
	are isomorphisms.
\end{rem}
\begin{ex} 
		For $\mathfrak{g}=\text{sl}_3(\mathbb{C})$ and $\ii=(1,2,1)$ we have
\begin{align*}
\gCAid\left(\lambda_1,\lambda_2,x_1,x_2,x_3\right)&=\left( \frac{x_1x_3^2}{\lambda_1x_2},\frac{x_2}{\lambda_2x_3},\frac{x_2}{x_1x_3},\frac{x_3}{x_2},\frac{1}{x_3}\right),\\
\gioti\left(x_{-1},x_{-2},x_1,x_2,x_3\right)&=\left(\frac{1}{x_3},\frac{1}{x_2},\frac{x_{-1}}{x_1}, \frac{x_{-2}}{x_2},\frac{x_1}{x_3}\right).
\end{align*}
\end{ex}
The families $(\gCAid)$ and ($\gioti$) have the following transformation behaviour.
\begin{lem}\label{transstringcluster}
		For $\ii, \jj\in\W$ the following diagrams commutes.
	$$\xymatrix{
		\TTia \ar[dd]^{\trl^{\ii}_{\jj}} \ar[rr]^{\gioti} & & \gr\TTis \ar[rr]^{\gCAid}\ar[dd]_{\text{id} \times \trs_{\jj}^{\ii}} & &\TTix \ar[dd]^{\trs^{\ii}_{\jj}} \\
		& \\
		{\TTiap} \ar[rr]^{\gioti} & & \gTTjs \ar[rr]^{\gCAjd} & & \TTixp.}
	$$
\end{lem}
\begin{proof} 
	Without loss of generality we can assume that $\jj$ is obtained from $\ii$ by a $3$-move at position $k$ with $i_{k}=a$.
Then by Lemma \ref{braid1}, the tori $\TTixp$ and $\TTiap$ are obtained from $\TTix$ and $\TTia$ by $\mathcal{X}$-cluster and $\mathcal{A}$-cluster transformation at the $4$-valent vertex $k-1$ of $\Gamma_{\ii}$, respectively. The commutativity of the left diagram then can be checked by direct computation. The commutativity of the right diagram is obtained by direct computation using Lemma \ref{diaglev} and that for $a\in I$
	$$
	\prod_{r=1}^{m_a} (\gCAid (\lambda, x))_{a,r} = \lambda_a^{-1}.
	$$
 \end{proof}

We relate the GHKK- and the BK-decoration function to the functions $\strpotia$ and $\strpotiald$ introduced in Definition \ref{nu} and \ref{nulld}, respectively:
\begin{thm}\label{string2}  For $a\in \pm I$ and $\ii \in \W$ we have 
	\begin{align}\label{strpot=clustera}
	\stlpotia&=	\restr{W_{a}}{\TTix}\circ \gCAid,\\
	\label{strpot=cluster}
	\restr{f^B_{a}}{\TTia}&={\stlpotiald}\circ \gioti.
	\end{align}
\end{thm}
\begin{proof}
	For $a>0$ equalities \eqref{strpot=clustera} and \eqref{lustigpot=cluster} follow directly from Lemma \ref{transstringcluster} together with Proposition \ref{locpotGHKK1} and Proposition \ref{apotex}, respectively.
	 We define $k(a,s)$ by $v_{k(a,s)}:=v_{a,s}$. 
For $a<0$ equality \eqref{strpot=clustera} follows from Proposition \ref{locpotGHKK1} and 
$$
\prod_{r=0}^s \prod_{\ell\in [N]} x_{\ell}^{\sym{k(-a,r)}{\ell}} = x_{-a,s+1} \prod_{\ell \leq k(-a,s+1)} x_{\ell}^{-c_{i_\ell, -a}}.
$$
To prove \eqref{strpot=cluster} we compute using Proposition \ref{apotex} 
$$\stlpotiald\circ\gioti  (x)=\displaystyle\left(\prod_{b\in I}x^{-c_{-a,b}}_{b,m_b}\right)\displaystyle\sum_{\substack{k\in[N] \\ i_k=-a}}\frac{x_k}{x_{k^-}}\displaystyle\prod_{\ell=k+1}^{N} \left(\frac{x_{\ell^-}}{x_{\ell}}\right)^{-c_{i_{\ell}, -a}} 
=\restr{f^{B}_{a}}{\TTia}(x).$$
\end{proof}
As a direct corollary of Theorem \ref{string2} we obtain:
\begin{cor}\label{stringreg} For $a\in I$ the functions $\stlpotia$ and $\stlpotiald$ are regular.
\end{cor}
Theorem \ref{string2} has the following implications for the interplay of parametrizations of canonical bases.
\begin{cor}\label{stringthm}
	We have the following equality of cones:
	\begin{align}\label{stringconesuperpot}
	\xci &=\tro{\gCAid} (\gsci),\\ 
	\gscild&=\tro{\gioti} \aci. \notag
	\end{align}
\end{cor}
\begin{proof}
	The statement follows by tropicalization Theorem \ref{string2} and Proposition \ref{wstringpos}.
\end{proof}
\begin{rem}
	For the special case that $\mathfrak{g}$ is of type $A$, Equality \eqref{stringconesuperpot} appeared in \cite{M} for the lexicographically minimal reduced word and in \cite{BF} for an arbitrary reduced word.
\end{rem}

\section{unimodularity of cones and polytopes}
In this section we view all cones as subsets of $\mathbb{R}^{n+N}$ and will refer to them by the same name as their integral analogs.

Let $m\in \mathbb{N}$ and $C_1,C_2\subset \mathbb{R}^m$ be two polyhedral cones. We call a bijection $f:C_1 \rightarrow C_2$ a \emph{unimodular isomorphism} of $C_1$ and $C_2$ if there exists a lattice isomorphism $g:\mathbb{Z}^m \rightarrow \mathbb{Z}^m$ such that $f=\restr{g}{C_1}$.

\begin{lem}\label{unicones} 
We have the following unimodular isomorphisms: 
\begin{align}
\tro{\giotid}&: \glci \rightarrow \xci, \label{firstequi}\\
\tro{\gCAid}&: \gsci \rightarrow \xci, \label{secequi}\\
\tro{\gCAi}&: \aci  \rightarrow \glcild, \label{thirdequi}\\
\tro{\gioti}&: \aci\rightarrow \gscild  .\label{fourequi}
\end{align}		
\end{lem}
\begin{proof} 
Reordering the coordinates on $x\in[\gr\TTil]_{trop}\otimes_{\mathbb{Z}}\mathbb{R}$ as \begin{equation*}
(x_1,\ldots,x_N,\lambda_1,\ldots,\lambda_n)
\end{equation*} the definition of $\tro{\giotid}$ yields that the corresponding matrix has integer entries, is lower triangular and all diagonal entries are equal to $-1$, whereas the definition of $[\gr\CAi]_{trop}$ yields that the corresponding matrix has integer entries, is upper triangular and all diagonal entries are equal to $-1$. Hence $\tro{\giotid}$ and $[\gr\CAi]_{trop}$ are lattice automorphisms of $\mathbb{Z}^{n+N}$. Using Corollary \ref{lusthm}, Claim \ref{firstequi} and \ref{thirdequi} follow.

Reordering the coordinates on $x\in[\gr\TTis]_{trop}\otimes_{\mathbb{Z}}\mathbb{R}$ as $$(\lambda_1,\ldots,\lambda_n,x_1,\ldots,x_N)$$ the definition of $[\gr\CAid]_{trop}$ yields that the corresponding matrix has integer entries, is upper triangular and all diagonal entries are equal to $-1$. Hence $[\gr\CAid]_{trop}$ is a lattice automorphism of $\mathbb{Z}^{n+N}$. Using Corollary \ref{stringthm} Claim \ref{secequi} follows.

Reordering the coordinates on $y\in[\TTis]_{trop}\otimes_{\mathbb{Z}}\mathbb{R}$ as $$(x_1,\ldots,x_N,\lambda_n,\lambda_{n-1},\ldots,\lambda_1)$$ the definition of $\tro{\gioti}$ yields that the corresponding matrix has integer entries, is upper triangular and all diagonal entries in $\{-1,1\}$. Hence $\tro{\gioti}$ is a lattice automorphism of $\mathbb{Z}^{n+N}$. Using Corollary \ref{stringthm} Claim \ref{fourequi} follows.

\end{proof}
Using Lemma \ref{unicones} we deduce a unimodular isomorphism of the graded string cone and the graded cone of Lusztig's parametrization which can be found in the literature combining \cite[Corollaire 3.5]{MG}, \cite[Lemma 6.3]{CMMG}.
\begin{prop}\label{unilusstring} \begin{enumerate}
\item The map 
$\tro{\giotid^{-1} \circ \gCAid}$ is a unimodular isomorphism of $\gsci$ and $\glci$. Explicitly, it given by
\begin{align}
\begin{split}
(\lambda',x')&=\tro{\giotid^{-1} \circ \gCAid} (\lambda,x),\\
\lambda'_a&=\lambda_{a^*},\\
x'_k &=\lambda_{i_k} -x_k - \displaystyle\sum_{\ell > k} c_{i_k, i_{\ell}}x_{\ell}.\label{classuni}
\end{split}
\end{align}
\item 
The map $\tro{\gCAi\circ \gioti^{-1}}$ is a unimodular isomorphism of $\gscild$ and $\glcild$. Explicitly, it is given by
\begin{align}\begin{split}
(\lambda',x')&=\tro{\gCAi \circ \gioti^{-1}} (\lambda,x),\\
\lambda'_a&=\lambda_{a^*},\\
x'_k &=(\sum_{a\in I}c_{i_k,a}\lambda_a) -x_k - \displaystyle\sum_{\ell > k} c_{i_k, i_{\ell}}x_{\ell}. \label{langlandsdualuni}
\end{split}
\end{align}
\end{enumerate}
\end{prop}
\begin{proof}
The unimodularity follows from Lemma \ref{unicones}. It remains to show the explicit description of the maps. We define $k(a,s)$ by $v_{k(a,s)}:=v_{a,s}$. Then \eqref{classuni} follows by tropicalizing the identities
\begin{align*}
x'_{a,r} &= \prod_{r=0}^{s-1} (\tro{\gCAid} (\lambda, x))_{a,r}^{-1} = \lambda_a  x_{a,s}^{-1} \prod_{\ell> k(a,s)} x_\ell^{-c_{i_\ell, a}},\\
\lambda_a' &=\prod_{r=0}^{m_{a^*}} (\tro{\gCAid} (\lambda, x))_{a^*,r}^{-1} = 
\lambda_{a^*}.
\end{align*}

Using $\gioti^{-1}(\lambda,x)=\lambda_{a}^{-1}\displaystyle\prod_{s>r}x_{a,s}$ Equality \eqref{langlandsdualuni} follows by applying $\gCAi$ and tropicalizing.
\end{proof}

We call two polytopes $P_1,P_2 \subset \mathbb{R}^m$ \emph{affine unimodular isomorphic} if there exists a lattice isomorphism $g: \mathbb{Z}^m \rightarrow \mathbb{Z}^m$ and a vector $v\in \mathbb{Z}^m$ such that $g(P_1)+v=P_2$.

For $\lambda \in \mathbb{N}^I$ the polytope
$$P\sci(\lambda)=\{x\in \mathbb{N}^N \mid (\lambda,x)\in \gr\sci\}$$
is called the \emph{string polytope of weight $\lambda$}.
For $\lambda \in \mathbb{N}^I$ the polytope
$$P\lci (\lambda)=\{x\in \mathbb{N}^N \mid (\lambda,x)\in \glci\}$$
is called \emph{Lusztig's polytope of weight $\lambda$}.

\begin{prop} 
For $\lambda=(\lambda_1,\ldots,\lambda_n)\in \mathbb{N}^{N}$ and \\ $\lambda^*:=(\lambda_{1^*},\ldots, \lambda_{n^*})$, the polytopes $P\sci(\lambda)$ and $P\sci(\lambda^*)$ are affine unimodular isomorphic.
\end{prop}
\begin{proof}
Note that $$(\lambda,P\sci(\lambda))=\{(\lambda,x) \mid x \in P\sci(\lambda)\}\subset \gr\sci$$ and $$(\lambda^*,P\lci (\lambda^*))=\{(\lambda^*,x) \mid x \in P\lci (\lambda^*)\}\subset \glci.$$
By Proposition \ref{unilusstring}, we have $[(\gr\iota_{\ii}^{*})^{-1}\circ \CAid]_{trop}(\lambda,P\sci(\lambda))=(\lambda^*,P\lci (\lambda^*))$ and the claim follows.
\end{proof}

\section{Proof of Theorem \ref{thma}}
In this section we prove Theorem \ref{thma} describing explicitly the relation of the GHKK-potential and the decoration function.

\begin{proof}[Proof of Theorem \ref{thma}]
	Let $\mathcal{D}: \mathbb{G}_m^{I}\times \mathbb{G}_m^{N} \rightarrow \mathbb{G}_m^{I}\times \mathbb{G}_m^{N}$ be the regular map given by $\mathcal{D}(\lambda,x)=(\lambda',x)$ and
	$$\lambda'_a=\displaystyle\prod_{b\in I}\lambda^{c_{a,b}}_b.$$
For
$$
	\varphi_{\ii} :=  \gr\NAid \circ \mathcal{D}\circ \gr\CAi : \TTia \rightarrow \TTix
$$
we obtain by Theorem \ref{lempot} for $\ii\in\W$
	$$\restr{f_{B}}{\TTia}=\restr{W}{\TTix}\circ \varphi_{\ii}.$$

By \cite[Lemma 3.6]{Zel} the tori $(\TTia)_{\ii\in\W}$ cover $G^{\w0,e}$ up to codimension $2$. Consequently there is a regular map $\varphi: \mathcal{A} \rightarrow \mathcal{X}$ with $\restr{\varphi}{\TTia}=\varphi_{\ii}$ satisfying ${f_{B}}=W \circ \varphi.$
\end{proof}

Note that $\varphi$ is not injective.

By Theorem \ref{string2} we obtain that the map $\varphi':\mathcal{A} \rightarrow \mathcal{X}$ defined by 
$$
\varphi '_{\ii}:=\gr\CAid \circ \mathcal{D}\circ \gr\NAi : \TTia \rightarrow \TTix 
$$
also satisfies ${f_{B}} = W \circ \varphi'.$ We deduce from a straightforward geometric version of Proposition \ref{unilusstring} that 
$$\gr\CAid \circ \gr\NAi = \gr\NAid \circ \gr\CAi.$$ 
Since $\mathcal{D}$ commutes with $\gioti$ and $\gCAi$ we conclude $\varphi_{\ii}'=\varphi_{\ii}$ and thus by \cite[Lemma 3.6]{Zel} the functions $\varphi'$ and $\varphi$ coincide.

\begin{rem}
In \cite{Ge} the first author introduces a construction of the map $\varphi$ in type $A$ as a canonical modification of the $\pp$-map $\pp : \mathcal{A} \rightarrow \mathcal{X}$ given in \cite[Equation (6)]{FG}. In \cite{Ge} the map $\varphi$ occurs as a crucial ingredient in the explicit definition of a $B(\infty)$-crystal structure on the analogues of $\aci\subset \tro{\TTia}$ and $\xci\subset \tro{\TTix}$ for the reduced double Bruhat cell.
\end{rem}

\def\cprime{$'$} \def\cprime{$'$} \def\cprime{$'$} \def\cprime{$'$}

\end{document}